\pdfoutput=1 %for arXiv submissions
\RequirePackage[l2tabu, orthodox]{nag} %before \documentclass
\documentclass[a4paper,reqno,11pt]{article} %default font size: 10pt %equation numbers on the right
\usepackage[stretch=10]{microtype} %micro layout engine
\usepackage[german,french,UKenglish]{babel}
\usepackage[T1]{fontenc} %needed for French
\usepackage[utf8]{inputenc} %before csquotes
\usepackage{csquotes} %quotations in different languages
\usepackage{letltxmacro} %to fix the parskip before proofs
\usepackage{amsmath,amsfonts,amssymb,amsthm,mathrsfs}
\usepackage[margin=1.3in]{geometry}
\usepackage{setspace}
\usepackage{enumitem}
\usepackage{calc}
\usepackage{makebox}
\usepackage{xifthen}
\usepackage{titlesec} %for section title spacing - breaks nameref in hyperref
\usepackage{tikz}
\usepackage{tikz-cd} %commutative diagrams in TikZ
\usetikzlibrary{calc}
\usetikzlibrary{matrix,arrows}
\usetikzlibrary{decorations.markings}
\usepackage[lf]{Baskervaldx} % lining figures
\let\rtimess\rtimes
\usepackage[bigdelims,vvarbb]{newtxmath} % math italic letters from Nimbus Roman
\let\rtimes\rtimess %looks weird in newtxmath
\usepackage[cal=boondoxo]{mathalfa} % mathcal from STIX, unslanted a bit

\usepackage{verbatim} %includes multiline comment
\usepackage{aliascnt} %for creating different hyperref references for different theoremstyles
\usepackage[unicode,colorlinks=true,urlcolor=blue!70!black,citecolor=blue!60!black,linkcolor=blue!60!black,pdfauthor={Remy van Dobben de Bruyn},pdftitle={Grothendieck Galois theory and étale exodromy}]{hyperref} %always add last

\titlespacing{\section}{0pt}{12pt plus 4pt minus 2pt}{0pt plus 2pt minus 2pt}
\titlespacing{\subsection}{0pt}{0pt plus 2pt minus 2pt}{0pt plus 2pt minus 2pt}
\titlespacing{\subsubsection}{0pt}{0pt plus 2pt minus 2pt}{0pt plus 2pt minus 2pt}

\titleformat{\section}[block]{\Large\bfseries\scshape\filcenter}{\thesection.}{1ex}{}
\titleformat{\subsection}{\large\scshape\filcenter}{\thesubsection}{1ex}{}

\numberwithin{equation}{section}
\newtheoremstyle{thms}{1em}{0pt}{\itshape}{}{\itshape\bfseries}{. ----}{ }{\thmname{#1} \thmnumber{#2}}
\theoremstyle{thms}
\newaliascnt{Thm}{equation}
\newtheorem{Thm}[Thm]{Theorem}				%Theorem
\aliascntresetthe{Thm}

\newaliascnt{Prop}{equation}							%Proposition

\aliascntresetthe{Prop}

\newaliascnt{Lemma}{equation}						%Lemma
\newtheorem{Lemma}[Lemma]{Lemma}
\aliascntresetthe{Lemma}

\newaliascnt{Cor}{equation}							%Corollary
\newtheorem{Cor}[Cor]{Corollary}
\aliascntresetthe{Cor}

\newaliascnt{Conj}{equation}							%Conjecture

\aliascntresetthe{Conj}

\newaliascnt{Question}{equation}						%Question

\aliascntresetthe{Question}

\newtheoremstyle{defs}{1em}{0pt}{}{}{\itshape\bfseries}{. ----}{ }{\thmname{#1} \thmnumber{#2}}
\theoremstyle{defs}
\newaliascnt{Rmk}{equation}							%Remark

\aliascntresetthe{Rmk}

\newaliascnt{Fact}{equation}							%Fact

\aliascntresetthe{Fact}

\newaliascnt{Def}{equation}							%Definition
\newtheorem{Def}[Def]{Definition}
\aliascntresetthe{Def}

\newaliascnt{Ex}{equation}							%Example
\newtheorem{Ex}[Ex]{Example}
\aliascntresetthe{Ex}

\newaliascnt{Con}{equation}							%Construction

\aliascntresetthe{Con}

\newaliascnt{Not}{equation}							%Notation

\aliascntresetthe{Not}

\newaliascnt{Setup}{equation}						%Setup

\aliascntresetthe{Setup}

\newaliascnt{Picture}{equation}						%Picture

\aliascntresetthe{Picture}

\newtheoremstyle{par}{1em}{0pt}{}{}{\itshape\bfseries}{. ----}{ }{\thmnumber{#2}}
\theoremstyle{par}
\newaliascnt{Par}{equation}							%Paragraph
\newtheorem{Par}[Par]{Paragraph}
\aliascntresetthe{Par}

\theoremstyle{thms}
\newtheorem{thm}{Theorem}
\newtheorem*{thm*}{Theorem}

\newtheorem*{lemma*}{Lemma}

\LetLtxMacro\oldproof\proof	%Proof without extra skip before and after
\renewcommand{\proof}[1][Proof]{\oldproof[#1]\unskip}
\LetLtxMacro\oldendproof\endproof
\renewcommand{\endproof}{\oldendproof\unskip}

\newenvironment{itemize*} %no longer needed because of \setlist below?
  {\begin{itemize}
    \setlength{\itemsep}{1em}
    \setlength{\parskip}{-1em}
    \setlength{\topsep}{0pt}
    \setlength{\partopsep}{0pt}}
  {\end{itemize}}
  
\newenvironment{enumerate*}
  {\begin{enumerate}
    \setlength{\itemsep}{1em}
    \setlength{\parskip}{-1em}
    \setlength{\topsep}{0pt}
    \setlength{\partopsep}{0pt}}
  {\end{enumerate}}
  
\setlist{itemsep=0em,topsep=0cm,partopsep=0em,parsep=\lineskip}

\setlist[enumerate]{label=\normalfont(\alph*), align=left, leftmargin=3em, labelwidth=1em, itemindent=0pt, listparindent=0pt, labelindent=1em, labelsep=*}

\setlist[itemize]{leftmargin=1.3em}

\tikzcdset{arrow style=math font}

\newcommand{\Set}{\mathbf{Set}}
\newcommand{\Fin}{\mathbf{Fin}}
\newcommand{\ProFin}{\mathbf{ProFin}}
\renewcommand{\Top}{\mathbf{Top}}

\newcommand{\et}{{\operatorname{\acute et}}}

\newcommand{\lc}{_{\operatorname{lc}}}
\newcommand{\cons}[1][]{_{\ifthenelse{\equal{#1}{}}{\operatorname{cons}}{#1\operatorname{-cons}}}}

\newcommand{\coh}{^{\operatorname{coh}}}
\newcommand{\eff}{^{\operatorname{eff}}}

\DeclareMathOperator{\Pro}{Pro}
\DeclareMathOperator{\Hom}{Hom}

\DeclareMathOperator{\Sub}{Sub}
\DeclareMathOperator{\Fun}{Fun}

\DeclareMathOperator{\ob}{ob}
\DeclareMathOperator{\Sh}{Sh}

\DeclareMathOperator*{\colim}{colim}

\DeclareMathOperator{\ev}{ev}
\DeclareMathOperator{\el}{el}

\newcommand{\op}{^{\operatorname{op}}}
\newcommand{\cts}{^{\operatorname{cts}}}
\newcommand{\lex}{^{\operatorname{lex}}}
\newcommand{\fin}{^{\operatorname{fin}}}

\DeclareMathOperator{\id}{id}

\newcommand{\punct}[1]{\makebox[0pt][l]{\,#1}} %so that full stops do not mess with tikz layout

\let\OLDthebibliography\thebibliography
\renewcommand\thebibliography[1]{
  \OLDthebibliography{#1}
  \setlength{\parskip}{0pt}
  \setlength{\itemsep}{0pt plus 0.1em}
}

\begin{document}

\renewcommand{\sectionautorefname}{Section}
\renewcommand{\subsectionautorefname}{Subsection}		%Subsection

\begin{center}
\vspace*{-3em}
\noindent\makebox[\linewidth]{\rule{15cm}{0.4pt}}
\vspace{-.5em}

{\LARGE{\textbf{Grothendieck Galois theory and  \'etale exodromy}}}

\vspace{.5em}

{\textsc{Remy van Dobben de Bruyn}}

\vspace*{.25em}
\leavevmode
\noindent\makebox[\linewidth]{\rule{11cm}{0.4pt}}
\vspace{.5em}
\end{center}

\renewcommand{\abstractname}{\small\bfseries\scshape Abstract}

\begin{abstract}\noindent
Finite \'etale covers of a connected scheme $X$ are parametrised by the \'etale fundamental group via the monodromy correspondence. This was generalised to an \emph{exodromy correspondence} for constructible sheaves, first in the topological setting by MacPherson, Treumann, Lurie, and others, and recently also for \'etale and pro-\'etale sheaves by Barwick--Glasman--Haine and Wolf. The proof of the \'etale exodromy theorem is long and technical, using many new definitions and constructions in the $(\infty,2)$-category of $\infty$-topoi. This paper gives a quick proof of the \'etale exodromy theorem for constructible sheaves of sets (with respect to a fixed stratification), in the style of Grothendieck's Galois theory.
\end{abstract}

\section*{Introduction}
Let $X$ be a qcqs scheme with a finite decomposition $X = \coprod_{s \in S} X_s$ into connected constructible locally closed subsets. Recall that an \'etale sheaf $\mathscr F$ on $X$ is \emph{$S$-constructible} if the restriction to each stratum $X_s$ is locally constant with finite stalks. Choose a geometric point $\bar s \in X_s$ in each stratum, and let $F_{\bar s} \colon \Sh\cons[S](X_\et) \to \Fin$ be the corresponding fibre functor. Let $\Pi_1^\et(X,S)$ be the full subcategory of $\Fun(\Sh\cons[S](X_\et),\Fin)$ on the fibre functors $F_{\bar s}$ for $s \in S$, which is naturally enriched in profinite sets (see \autoref{Ex Funlex}). We prove the following:

\begin{thm}
Let $X$ be a qcqs scheme with a finite constructible decomposition $X = \coprod_{s \in S} X_s$. Assume that each stratum $X_s$ is connected, and let $\Pi_1^\et(X,S)$ be as above. Then the functor
\begin{align*}
\ev \colon \Sh\cons[S](X_\et) &\to \Fun\cts\big(\Pi_1^\et(X,S),\Fin\big) \\
\mathscr F &\mapsto \big(F \mapsto F(\mathscr F)\big)
\end{align*}
is an equivalence.
\end{thm}

We set up an abstract framework for such a correspondence to hold, generalising Grothendieck's Galois theory \cite[Exp.\ V, Thm.\ 4.1]{SGA1}. This gives the following variants:
\begin{itemize}
\item If $Y_s \to X_s$ is a connected profinite \'etale Galois cover for each $s \in S$, we may restrict to the $S$-constructible sheaves that are trivialised over $Y_s$ for every $s$.
\item In particular, if $X$ is a scheme of finite type over a field $k$ of positive characteristic, we can restrict to the constructible sheaves $\mathscr F \in \Sh\cons[S](X_\et)$ such that the locally constant sheaf $\mathscr F|_{X_s}$ is tamely ramified for all $s \in S$. This gives a \emph{tame fundamental category} $\Pi_1^t(X,S)$.
\end{itemize}
Since tame fundamental groups are expected to be finitely presented, there is a chance that $\Pi_1^t(X,S)$ might be computable in certain cases. In \cite{vDdB} we give an explicit computation of $\Pi_1^t(X,S)$, where $X$ is a split normal toric variety over a field $k$, and $S$ is the stratification by torus orbits.

\phantomsection
\subsection*{Idea of the proof}
The contemporary way to prove an exodromy theorem for $S$-constructible sheaves is by studying representability of fibre functors $F_{\bar s} \colon \Sh\cons[S](X) \to \Set$. If each $F_{\bar s}$ is represented by some object $\mathscr F_s \in \Sh\cons[S](X)$, then these sheaves generate the topos $\Sh\cons[S](X)$. Then use that the $\mathscr F_s$ are atomic \cite[Tag \href{https://kerodon.net/tag/03WY}{03WY}]{Kerodon}, \cite[Def.\ 2.2.10]{HainePortaTeyssier} or in a certain sense `projective' \cite[Lem.\ 2.1 and Cor.\ 2.2]{vDdB}, \cite{vDdBWolf} to formally deduce that they freely generate $\Sh\cons[S](X)$, i.e.\ that the Yoneda functor
\[
\Sh\cons[S](X) \to \Fun\big(\{\mathscr F_s\}\op,\Set\big) \simeq \Fun\big(\{F_{\bar s}\},\Set\big)
\]
is an equivalence. This is analogous to the proof of monodromy via universal coverings: the fibre functor $F_{\bar x} \colon \Sh\lc(X) \to \Set$ on locally constant sheaves is represented by the universal cover $\widetilde X \to X$, so $\pi_1(X) = F_{\bar x}(\widetilde X) = \Hom(h_{\widetilde X},F_{\bar x}) = \Hom(F_{\bar x},F_{\bar x})$ by the Yoneda lemma.

For Galois categories $\mathscr C$ (\autoref{Def Galois}) such as $\Sh\cons[S](X_\et)$, the fibre functors are left exact, so they are always \emph{pro-representable} since $\Pro(\mathscr C) \simeq \Fun\lex(\mathscr C,\Fin)\op$. Thus, we do not have to study the representability question, at the expense of getting an object that no longer lives in $\mathscr C$. In section 1, we define a profinite topology on the $\Hom$ sets in $\Fun\lex(\mathscr C,\Fin)$ (of which $\Pi_1(\mathscr C)$ will be a full subcategory), and show that the evaluation $\ev \colon \mathscr C \to \Fun(\Fun\lex(\mathscr C,\Fin),\Fin)$ lands in \emph{continuous} functors.

Inspired by \cite[Tag \href{https://stacks.math.columbia.edu/tag/03A0}{03A0}]{Stacks}, to show that $\ev \colon \mathscr C \to \Fun\cts(\Pi_1(\mathscr C),\Fin)$ is an equivalence, one needs to check that it is `locally fully faithful' and `locally essentially surjective', in the sense of criteria (3), (4), and (5) of [\emph{loc.\ cit.}]. We prove these statements simultaneously by the following method:
\begin{itemize}
\item Fix $F \in \Fun\cts(\Pi_1(\mathscr C),\Fin)$. Since $\Pi_1(\mathscr C)$ is finite and each $F(x)$ for $x \in \Pi_1(\mathscr C)$ is finite, there is a surjection $h_{\coprod x} = \coprod_i h_{x_i} \twoheadrightarrow F$ for some $x_1,\ldots,x_n \in \Pi_1(\mathscr C)$. Note that $h_{x_i}$ and $h_{\coprod x}$ live in $\Fun\cts(\Pi_1(\mathscr C),\ProFin)$, not $\Fun\cts(\Pi_1(\mathscr C),\Fin)$.
\item The $x_i$ are pro-representable in $\mathscr C$, so $h_{\coprod x}$ is a cofiltered limit of $\ev(A_j)$ for $A_j \in \mathscr C$.
\item But $F$ is opcompact in $\Fun\cts(\Pi_1(\mathscr C),\ProFin)$, meaning that the natural map
\[
\colim_{\substack{\longrightarrow \\ j}} \Hom\big(\ev(A_j),F\big) \to \Hom\Big(\lim_{\substack{\longleftarrow \\ j}} \ev(A_j),F\Big) = \Hom\big(h_{\coprod x},F\big)
\]
is an isomorphism. This gives an epimorphism $\ev(A) \twoheadrightarrow F$, and is also used to show that any map $\ev(A) \to \ev(B)$ locally comes from a map in $\mathscr C$.
\end{itemize}
Opcompactness of $F$ is proved in \autoref{Lem finite opcompact}, and the rest of the argument outlined above is carried out in \autoref{Lem pointed filtered}. This implies local fullness and local essential surjectivity by \autoref{Cor filtered colimit}, and the main theorem \autoref{Thm Galois exodromy} follows easily.

\phantomsection
\subsection*{Comparison with earlier proofs}
The \'etale exodromy theorem \cite{BGH} is in many ways much stronger than our result: it deals with constructible sheaves of spaces instead of constructible sheaves of sets, it works for all stratifications simultaneously, and there is no connectedness hypothesis on the fibres. The proof (very roughly) consists of the following steps:
\begin{itemize}
\item First prove a monodromy theorem in this setting. For a connected scheme, this is more or less \'etale homotopy theory. But $X$ may have a profinite set of connected components, in which case the result is attributed to Lurie's profinite shape theory \cite[Appendix E]{LurieSAG}.
\item Treat the case of two strata using the vanishing $\infty$-topos (a certain lax fibre product in $\infty$-topoi), via a base change theorem similar to results of Gabber.
\item Treat finitely many strata by induction, using a Segal condition to extract deeper vanishing $\infty$-topoi from the classical one (for a decomposition $X = U \amalg Z$).
\item The result for all constructible sheaves follows by taking the limit over all stratifications.
\end{itemize}
We see that our proof is a lot more direct: there is no induction on the number of strata, no glueing, no lax fibre products or Segal conditions, and so on. In fact, the partial order on $S$ plays no role in our argument: the statement and proof only depend on the collection of strata. This removes subtle issues about hypotheses on the stratification:

\textbf{\textit{Example.}} ---- If $X$ is a union of two irreducible curves $C$ and $D$ meeting at two points $x$ and $y$, then the decomposition into the locally closed subsets $X_0 = C \setminus \{x\}$ and $X_1 = D \setminus \{y\}$ is perfectly allowed. But the only map $X \to \{0,1\}$ that is continuous is when $\{0,1\}$ has the indiscrete topology, meaning that $X$ is stratified over a preorder instead of a poset.

In addition, our data type is more pragmatic for computations. The output of \cite{BGH} is a limit of finite (layered) categories. Writing such a thing down is impractical, and it is even less clear what a continuous functor to $\Set$ is in that case. On the other hand, a category with finitely many objects and profinite $\Hom$ sets is something that can actually be computed in examples, such as \cite{vDdB}.

In work in progress of van Dobben de Bruyn and Wolf \cite{vDdBWolf}, we will also give a more direct proof of the $\infty$-categorical \'etale and pro-\'etale exodromy theorems (with fixed or arbitrary stratification), using condensed categories instead of profinite categories.

\phantomsection
\subsection*{Acknowledgements}
This work was greatly influenced by discussions with Sebastian Wolf, Peter Haine, Clark Barwick, and Mauro Porta over the past years. The author is supported by the NWO grant VI.Veni.212.204.

\section{Categories enriched in profinite sets}
Recall that a category enriched in profinite sets is a category $\mathscr C$ with a profinite topology on $\Hom(x,y)$ for $x, y \in \mathscr C$ such that the composition $\Hom(y,z) \times \Hom(x,y) \to \Hom(x,z)$ is continuous for all $x,y,z \in \mathscr C$.

\begin{Ex}\label{Ex Funlex}
Let $\mathscr C$ be a small and finitely complete category. Then the category $\Fun\lex(\mathscr C,\Fin)$ of left exact functors $F \colon \mathscr C \to \Fin$ is canonically enriched in profinite sets: given $F, G \in \Fun\lex(\mathscr C,\Fin)$, note that $\Hom(F,G) = \int_{A \in \mathscr C} \Hom(F(A),G(A))$ is the equaliser of the arrows
\begin{align*}
\prod_{A \in \mathscr C} \Hom\big(F(A),G(A)\big) &\rightrightarrows \prod_{f \colon A \to B} \Hom\big(F(A),G(B)\big)
\end{align*}
where each of the sets $\Hom(F(A),G(A))$ is finite. For $F, G, H \in \Fun\lex(\mathscr C,\Fin)$, composition $\Hom(G,H) \times \Hom(F,G) \to \Hom(F,G)$ is continuous since the same goes for the map $\prod_A \Hom(G(A),H(A)) \times \prod_A \Hom(F(A),G(A)) \to \prod_A \Hom(F(A),H(A))$.

Concretely, clopen subsets of $\Hom(F,G)$ are the sets $U_A(S) = \{\phi \in \Hom(F,G)\ |\ \phi_A \in S\}$ for $A \in \mathscr C$ and $S \subseteq \Hom(F(A),G(A))$. These sets indeed form a Boolean lattice in $\Hom(F,G)$: they are clearly closed under complements, and if $F(A), F(B) \neq \varnothing$, then $U_A(S) \cap U_B(T) = U_{A \times B}(S \times T)$, where $S \times T \subseteq \Hom(F(A) \times F(B),G(A) \times G(B))$ denotes the set $\{s \times t\ |\ s \in S, t \in T\}$. (If one of $F(A), F(B)$ is empty, there is nothing to prove.)

A basis for the topology is given by $U_A(x,y) = \{\phi \in \Hom(F,G)\ |\ \phi_A(x) = y\}$ for $A \in \mathscr C$ and $(x,y) \in F(A) \times G(A)$. Indeed, note that $U_A(x,y) \cap U_B(u,v) = U_{A \times B}\big((x,u),(y,v)\big)$, so the collection of the sets $U_A(x,y)$ is closed under finite intersection. They are clearly open, and conversely, the set $U_A(S)$ for $S \subseteq \Hom(F(A),G(A))$ agrees with
\[
\bigcup_{f \in S} \bigcap_{x \in F(A)} U_A\big(x,f(x)\big).
\]
In this notation, continuity of composition $\circ \colon \Hom(G,H) \times \Hom(F,G) \to \Hom(F,H)$ follows since $\circ^{-1}(U_A(x,z)) = \bigcup_{y \in G(A)} U_A(y,z) \times U_A(x,y)$ for any object $A \in \mathscr C$ and any $(x,z) \in F(A) \times H(A)$.
\end{Ex}

\begin{Def}\label{Def continuous}
Let $\Pi$ be a small category enriched in topological spaces. We say that a functor $F \colon \Pi \to \Top$ is \emph{continuous} if for all $x,y \in \mathscr C$, the map $\Hom(x,y) \times F(x) \to F(y)$ given by $(\phi,a) \mapsto \phi(a)$ is continuous.

If $F$ lands in locally compact Hausdorff spaces, then this is equivalent to the condition that $\Hom(x,y) \to \hom(F(x),F(y))$ is continuous for the compact-open topology. If $F$ lands in sets (viewed as discrete topological spaces), it is equivalent to the condition that for all $(a,b) \in F(x) \times F(y)$, the set $\{f \in \Hom(x,y)\ |\ F(f)(a) = b\}$ is open. We write $\Fun\cts(\Pi,\Top)$ for the category of continuous functors $\Pi \to \Top$, and likewise for $\Fun\cts(\Pi,\ProFin)$, $\Fun\cts(\Pi,\Fin)$, and $\Fun\cts(\Pi,\Set)$.
\end{Def}

\begin{Ex}\label{Ex continuous}
If $\mathscr C$ is a small category with finite limits, then the (small) category $\Fun\lex(\mathscr C,\Fin)$ is enriched in profinite sets by \autoref{Ex Funlex}. Thus, the representable functor $h_x = \Hom(x,-) \colon \Fun\lex(\mathscr C,\Fin) \to \ProFin$ is continuous for any $x \in \Fun\lex(\mathscr C,\Fin)$: this follows from the enrichment in $\ProFin$.
\end{Ex}

\begin{Lemma}\label{Lem finite opcompact}
Let $\Pi$ be a category with finitely many objects that is enriched in profinite spaces. Then every object $F \in \Fun\cts(\Pi,\Fin)$ is opcompact in $\Fun\cts(\Pi,\ProFin)$, i.e.\ $\Hom(-,F)$ takes cofiltered limits to filtered colimits.
\end{Lemma}

\begin{proof}
Let $\mathcal I$ be a filtered category, let $D \colon \mathcal I\op \to \Fun\cts(\Pi,\ProFin)$ be a diagram, and also write $D$ for its limit. We get a natural map
\begin{equation}
\colim_{\substack{\longrightarrow \\ i \in \mathcal I}} \Hom\big(D(i),F\big) \to \Hom\Big( \lim_{\substack{\longleftarrow \\ i \in \mathcal I\op}} D(i), F \Big) = \Hom(D,F).\label{Eq colim}
\end{equation}
Let $i \in \mathcal I$ and suppose $\alpha, \beta \colon D(i) \rightrightarrows F$ induce the same transformation $D \to F$. Since $\Fin$ is opcompact in $\ProFin$, for each $x \in \Pi$, there exists an arrow $i \to j_x$ in $\mathcal I$ such that the two arrows $D(j_x)(x) \to D(i)(x) \rightrightarrows F(x)$ agree. Since $\ob\Pi$ is finite, choosing $j$ with maps $j_x \to j$ for all $x \in \Pi$ shows that the maps $D(j) \to D(i) \rightrightarrows F$ agree, showing injectivity of \eqref{Eq colim}.

Conversely, if $\alpha \colon D \to F$ is a natural transformation, then for every $x \in \Pi$ there exists $j_x$ in $\mathcal I$ and a map $D(j_x)(x) \to F(x)$ inducing $D(x) \to F(x)$. Choosing $j$ with maps $j_x \to j$ for all $x \in \Pi$ gives an object $j \in \mathcal I$ such that each $\alpha_x \colon D(x) \to F(x)$ comes from a map $\alpha_{j,x} \colon D(j)(x) \to F(x)$ (which are not necessarily natural in $x$). For each $x \in \Pi$, write $H_x = \coprod_{y \in \Pi} \Hom(x,y) \in \ProFin$. For $x \in \Pi$ and $k \in \mathcal I_{j/}$, write $\alpha_{k,x}$ for the composition $D(k)(x) \to D(j)(x) \to F(x)$. Then the equaliser $E_{k,x}$ of the continuous maps
\[
\begin{tikzcd}[column sep=4em]
H_x \times D(k)(x) \ar{r}{\id \times \alpha_{k,x}}\ar{d} & H_x \times F(x) \ar{d} \\
\coprod\limits_{y \in \Pi} D(k)(y) \ar{r}[swap]{\coprod \alpha_{k,y}} & \coprod\limits_{y \in \Pi} F(y)
\end{tikzcd}
\]
is clopen, hence so is the complement $Y_{k,x}$. Thus $X_k = \coprod_{x \in \Pi} H_x \times D(k)(x)$ has a clopen decomposition by $E_k = \coprod_{x \in \Pi} E_{k,x}$ and $Y_k = \coprod_{x \in \Pi} Y_{k,x}$. The limit of $X_k$ over $k \in \mathcal I_{j/}\op$ is $\coprod_{x \in \Pi} H_x \times D(x)$, and the limit of $E_k$ is the disjoint union (over all $x \in \Pi$) of the equalisers of
\[
\begin{tikzcd}[column sep=4em]
H_x \times D(x) \ar{r}{\id \times \alpha_x}\ar{d} & H_x \times F(x) \ar{d} \\
\coprod\limits_{y \in \Pi} D(y) \ar{r}[swap]{\coprod \alpha_y} & \coprod\limits_{y \in \Pi} F(y)\punct{,}
\end{tikzcd}
\]
which equals $X_k$ by naturality of $\alpha$. Thus, $\lim_{k \in \mathcal I_{j/}\op} Y_k = \varnothing$, so some $Y_k$ must be empty already \cite[Tag \href{https://stacks.math.columbia.edu/tag/0A2R}{0A2R}]{Stacks}. This means $E_k = X_k$, so $\alpha_{k,x} \colon D(k) \to F$ is a natural transformation.
\end{proof}

\section{Galois categories}
We give a generalisation of Grothendieck's theory of Galois categories \cite[Exp.~V,~\S4]{SGA1} to categories with a finite family of jointly conservative fibre functors. The main result is \autoref{Thm Galois exodromy}, which says that any such category is equivalent to $\Fun\cts(\Pi,\Fin)$ for some profinitely enriched category $\Pi$ with finitely many objects. 
Note that even in the case where $\Pi$ has one object, it will only be a profinite monoid, not necessarily a group.

We begin with some recollections on pretopoi:

\begin{Par}
Recall \cite[Exp.\ VI, Exc.\ 3.11(a)]{SGA4II} that a \emph{pretopos} is a small category $\mathscr C$ such that 
\begin{enumerate}
\item $\mathscr C$ has finite limits,
\item $\mathscr C$ has finite coproducts, and finite coproducts are disjoint and universal,
\item every equivalence relation in $\mathscr C$ is effective,
\item every epimorphism in $\mathscr C$ is a universal effective epimorphism.
\end{enumerate}
A \emph{morphism of pretopoi} $\mathscr C \to \mathscr D$ is a functor $F \colon \mathscr C \to \mathscr D$ between pretopoi that preserves finite limits, finite coproducts, and effective epimorphisms.
\end{Par}

\begin{Par}
The category of pretopoi is equivalent to the opposite of the category of coherent topoi (with coherent geometric morphisms) \cite[Exp.\ VI, Exc.\ 3.11(d)]{SGA4II}. A pretopos $\mathscr C$ is sent to the topos $\Sh\eff(\mathscr C)$ of sheaves for the effective epimorphism topology: covering families $\{Y_i \to X\}$ are finite families such that $\coprod_i Y_i \to X$ is an effective epimorphism. Conversely, a coherent topos $\mathscr X$ is sent to the full subcategory $\mathscr X\coh \subseteq \mathscr X$ of coherent objects.
\end{Par}

\begin{Ex}\label{Ex pretopos}
If $\mathscr X$ is a coherent topos, then the full subcategory $\mathscr X\coh$ of coherent objects is a pretopos. In particular:
\begin{enumerate}
\item The category $\Fin$ is a pretopos, being the coherent objects in the coherent topos $\Set$.
\item If $X$ is a qcqs scheme, then $\Sh\cons(X_\et)$ is a pretopos, as these are the coherent objects in the coherent topos $\Sh(X_\et)$.
\item\label{item S-cons} If $X$ is a qcqs scheme and $X = \coprod_{s \in S} X_s$ is a finite constructible decomposition, then $\Sh\cons[S](X_\et)$ is a coherent topos: as a subcategory of $\Sh\cons(X_\et)$, it is stable under finite limits, finite coproducts, and effective epimorphisms.
\end{enumerate}
\end{Ex}

\begin{Def}
If $\mathscr C$ is a pretopos, then a \emph{fibre functor} is a morphism of pretopoi $F \colon \mathscr C \to \Fin$. Under the equivalence described above, these correspond to \emph{coherent} geometric morphisms $\Set \to \Sh\eff(\mathscr C)$.
\end{Def}

Thus, fibre functors correspond to coherent geometric points $\Set \to \Sh\eff(\mathscr C)$. This is a subtle difference with \cite{LurieUltra}, where all geometric points are needed to reconstruct the coherent topos $\Sh\eff(\mathscr C)$.

\begin{Def}\label{Def Galois}
A \emph{Galois category} is a pretopos $\mathscr C$ such that there exists a finite family $\{F_i\ |\ i \in I\}$ of fibre functors $F_i \colon \mathscr C \to \Fin$ that is jointly conservative, i.e.\ if $f \colon A \to B$ is a map in $\mathscr C$ such that $F_i(f)$ is an isomorphism for all $i \in I$, then $f$ is an isomorphism.
\end{Def}

\begin{Ex}\label{Ex scheme}
Let $X$ be a qcqs scheme, and let $X = \coprod_{s \in S} X_s$ be a finite decomposition into connected constructible locally closed subspaces. Then $\mathscr C = \Sh\cons[S](X_\et)$ is a pretopos by \autoref{Ex pretopos}\ref{item S-cons}. Choosing a geometric point $\bar s \to X$ in each fibre, we get fibre functors $F_{\bar s} \colon \mathscr C \to \Fin$, which are jointly conservative since the fibre functor $F_{\bar s} \colon \Sh\lc(X_{s,\et}) \to \Fin$ is conservative when $X_s$ is connected. Thus, $\mathscr C$ is a Galois category.
\end{Ex}\newpage

\begin{Ex}\label{Ex scheme with cover}
More generally, in the situation above, if $Y_s \to X_s$ is a connected Galois pro-\'etale cover for each $s \in S$, write $\mathscr C \subseteq \Sh\cons[S](X_\et)$ for the full subcategory on sheaves such that the restriction to $Y_s$ is constant for all $s \in S$. Then $\mathscr C$ is stable under finite limits, finite coproducts, and effective epimorphisms, so it is a pretopos. The same fibre functors $F_{\bar s}$ give a finite jointly conservative family of morphisms of pretopoi, so $\mathscr C$ is a Galois category. A particular case of interest is where $Y_s \to X_s$ is the maximal tame cover.
\end{Ex}

We start with some elementary observations about Galois categories.

\begin{Lemma}\label{Lem Galois elementary}
Let $\mathscr C$ be a Galois category with a finite conservative family of fibre functors $\{F_i\ |\ i \in I\}$.
\begin{enumerate}
\item\label{item faithful} The product $F = (F_i)_{i \in I} \colon \mathscr C \to \Fin^I$ is faithful. 
\item\label{item Hom finite} For every $A, B \in \mathscr C$, the set $\Hom(A,B)$ is finite.
\item\label{item Yon finite} The Yoneda embedding $\mathscr C \to \Fun\lex(\mathscr C,\Set)\op = \Pro(\mathscr C)$ sends $\mathscr C$ to $\Fun\lex(\mathscr C,\Fin)\op$.
\item\label{item sub finite} For every $A \in \mathscr C$, the poset $\Sub(A)$ of subobjects $B \hookrightarrow A$ up to isomorphism is finite.
\item\label{item effective epi} If $f \colon A \to B$ is a map in $\mathscr C$ such that $F(f)$ is injective (resp.\ surjective), then $f$ is a monomorphism (resp.\ an effective epimorphism).
\end{enumerate}
\end{Lemma}

\begin{proof}
By hypothesis, the product $F$ is left exact and conservative. Let $f, g \colon A \rightrightarrows B$ be parallel arrows with $F(f) = F(g)$, and let $E \to A$ be the equaliser. Then $F(E) \to F(A)$ is the equaliser of $F(f)$ and $F(g)$ by left exactness, hence $F(E) \to F(A)$ is an isomorphism as $F(f) = F(g)$. Conservativity shows that $E \to A$ is an isomorphism, proving \ref{item faithful}. Then \ref{item Hom finite} follows since $\Hom(A,B) \hookrightarrow \Hom(F(A),F(B))$ and Hom-sets in $\Fin^I$ are finite, and \ref{item Yon finite} is a reformulation of this.

Since $\mathscr C$ has finite limits, the poset $\Sub(A)$ is a meet-semilattice for any $A \in \mathscr C$. Since $F$ preserves monomorphisms, we get a map $F^* \colon \Sub(A) \to \Sub(F(A))$, which preserves meets since $F$ is left exact. If $F(B) = F(C)$ for $B, C \in \Sub(A)$, then the inclusions $B \cap C \to B$ and $B \cap C \to C$ become isomorphisms after applying $F$, so they are already isomorphisms by conservativity. Thus, $F^* \colon \Sub(A) \to \Sub(F(A))$ is injective, proving \ref{item sub finite} since $\Sub(F(A))$ is finite.

Since pretopoi are regular categories, we may choose a factorisation $A \twoheadrightarrow I \hookrightarrow B$ of $f$ into an effective epimorphism followed by a monomorphism. Since $F$ preserves finite limits and effective epimorphisms, we conclude that $F(A) \twoheadrightarrow F(I) \hookrightarrow F(B)$ is the image factorisation of $F(f)$. If $F(f)$ is injective (resp.\ surjective), then the map $F(A) \to F(I)$ (resp.\ $F(I) \to F(B)$) is an isomorphism, which by conservativity implies the same for $A \to I$ (resp.\ $I \to B$). This proves \ref{item effective epi}.
\end{proof}

\begin{Def}
Let $\mathscr C$ be a category with finite limits. Then the \emph{category of profinite objects of $\mathscr C$} is the full subcategory $\Fun\lex(\mathscr C,\Fin)\op \subseteq \Fun\lex(\mathscr C,\Set)\op$. It is denoted $\Pro\fin(\mathscr C) \subseteq \Pro(\mathscr C)$. Note that it is small if $\mathscr C$ is small.
\end{Def}

\begin{Def}
Let $\mathscr C$ be a Galois category, with a finite conservative family of fibre functors $\{F_i\ |\ i \in I\}$. Then the \emph{fundamental category} $\Pi_1(\mathscr C)$ is the full profinitely enriched subcategory of $\Fun\lex(\mathscr C,\Fin) = \Pro\fin(\mathscr C)\op$ on the fibre functors $\{F_i\ |\ i \in I\}$.

Write $\ev \colon \mathscr C \to \Fun\cts(\Pi_1(\mathscr C),\Fin)$ for the evaluation $A \mapsto \big(F_i \mapsto F_i(A))$. Under the inclusion $\Fun\cts(\Pi_1(\mathscr C),\Fin) \hookrightarrow \Fun\cts(\Pi_1(\mathscr C),\ProFin)$, it is given by the composition
\begin{equation}
\mathscr C \overset h\hookrightarrow \Pro\fin(\mathscr C) \overset h\hookrightarrow \Fun\cts\big(\Pro\fin(\mathscr C)\op,\ProFin\big) \to \Fun\cts\big(\Pi_1(\mathscr C),\ProFin\big)\label{Eq Yoneda Yoneda}
\end{equation}
of the Yoneda embedding of $\mathscr C$, the profinitely enriched Yoneda embedding of $\Pro\fin(\mathscr C)$, and the restriction along $\Pi_1(\mathscr C) \hookrightarrow \Pro\fin(\mathscr C)\op$.
\end{Def}

\begin{Par}
Recall that a functor $F \colon \mathscr C \to \Set$ is canonically isomorphic in $\Fun(\mathscr C,\Set)$ to the colimit
\[
\colim_{(A,x) \in \el(F)\op} h_A,
\]
where
\[
\el(F) = \Big(\{*\} \underset{\Set}\downarrow F\Big) \simeq \Big(h \underset{\Fun(\mathscr C,\Set)}\downarrow \{F\}\Big)\op
\]
is the category of elements of $F$. If $\mathscr C$ has finite limits, then $F$ is left exact if and only if $\el(F)$ is cofiltered, identifying $\Fun\lex(\mathscr C,\Set)\op$ with $\Pro(\mathscr C)$. As element of $\Pro(\mathscr C)$, a left exact functor $F$ is therefore identified with the limit
%h | {F}: objects are (A,\phi) with A \in \mathscr C\op and \phi \colon h_A \to F
%* | F: objects are (A,\phi) with A \in \mathscr C and \phi \colon * \to F(A).
%These are anti-equivalent by the Yoneda lemma.
\begin{equation}
F \cong \lim_{\substack{ \longleftarrow \\ (A,a) \in \el(F)}} h_A.\label{Eq limit}
\end{equation}
If all Hom-sets in $\mathscr C$ are finite, then each $h_A$ lives in $\Pro\fin(\mathscr C)$, so the same formula holds in $\Pro\fin(\mathscr C)$ for any left exact functor $F \colon \mathscr C \to \Fin$.
If $\mathscr C$ is a Galois category and $F = F_i$ is some object of $\Pi_1(\mathscr C)$, then the last two functors in \eqref{Eq Yoneda Yoneda} preserve all limits that exist, so we conclude that
\begin{equation}
h_{F_i} = \lim_{\substack{\longleftarrow \\ (A,a) \in \el(F_i)}} \ev(A).\label{Eq representable limit}
\end{equation}
In the locally constant case, we think of \eqref{Eq limit} as pro-representing the fibre functor by the universal cover. For constructible sheaves, it is useful to think of $F_{\bar s}$ as pro-represented by a `constructible Henselisation' at the universal cover $\widetilde{X_s} \to X_s$: the limit over all pairs $(\mathscr F,x)$ of an $S$-constructible sheaf $\mathscr F$ on $X$ with a section $x$ of the pullback of $\mathscr F$ to $\widetilde{X_s}$.
\end{Par}

\begin{Lemma}\label{Lem pointed filtered}
Let $\mathscr C$ be a Galois category, let $x = (x_1,\ldots,x_n) \in \Pi_1(\mathscr C)^n$, and write $h_{\coprod x} = \coprod_i h_{x_i} \in \Fun\cts\big(\Pi_1(\mathscr C),\ProFin\big)$. Then the category $\mathcal I = (\{h_{\coprod x}\} \downarrow \ev)$ of pairs $(A,a)$ with $A \in \mathscr C$ and $a = (a_1,\ldots,a_n) \in \ev(A)(x_1) \times \cdots \times \ev(A)(x_n)$ is cofiltered, and for every $F \in \Fun\cts\big(\Pi_1(\mathscr C),\Fin\big)$, the map
\begin{align}
\colim_{\substack{\longrightarrow \\ (A,a) \in \mathcal I\op}} \Hom\big(\ev(A),F\big) &\to \Hom(h_{\coprod x},F) \cong \prod_{i=1}^n F(x_i)\label{Eq pointed filtered} \\
\big(\phi \colon \ev(A) \to F\big) &\mapsto \phi(a) \nonumber
\end{align}
is an isomorphism.
\end{Lemma}

\begin{proof}
Viewing the $x_i$ as fibre functors $F_{x_i} \in \Fun\lex(\mathscr C,\Fin) = \Pro\fin(\mathscr C)\op$, their product $F_x = F_{x_1} \times \cdots \times F_{x_n}$ is left exact, so its category of elements $\el(F_x) \cong \mathcal I$ is cofiltered. Define the functor
\begin{align*}
\Psi \colon \el(F_{x_1}) \times \cdots \times \el(F_{x_n}) &\to \el(F_x)\\
\big((A_1,a_1),\ldots,(A_n,a_n)\big) &\mapsto \Big(\coprod_{i=1}^n A_i, \psi(a_1,\ldots,a_n)\Big),
\end{align*}
where the $i^{\text{th}}$ coordinate of $\psi \colon \prod_i F_{x_i}(A_i) \to \prod_i F_{x_i}(\coprod_j A_j)$ is induced by the insertion of $A_i$ into the coproduct. We will write objects of the left hand side as $(A,a)$, where $A = (A_1,\ldots,A_n) \in \mathscr C^n$, and $a = (a_1,\ldots,a_n) \in \prod_{i=1}^n F_{x_i}(A_i)$. Then $\Psi$ is given by $(A,a) \mapsto (\coprod A,\psi(a))$, where $\coprod A$ is shorthand for $\coprod_i A_i$.

We claim that $\Psi$ is coinitial: if $(A,a) \in \el(F_x)$ with $a = (a_1,\ldots,a_n)$, we need to check that $(\Psi \downarrow (A,a))$ is connected. Since $\Psi((A,a_1),\ldots,(A,a_n))$ admits a map to $(A,a)$ given by $A \amalg \ldots \amalg A \to A$, we see that $(\Psi \downarrow (A,a))$ is nonempty. Given objects $\Psi(B,b) \to (A,a)$ and $\Psi(C,c) \to (A,a)$ of $(\Psi \downarrow (A,a))$, define $D_i = B_i \times_A C_i$, and endow it with the points $d_i = (b_i,c_i) \in F_{x_i}(D_i) = F_{x_i}(B_i) \times_{F_{x_i}(A)} F_{x_i}(C_i)$. This gives an object $\Psi(D,d) \to (A,a)$ in $(\Psi \downarrow (A,a))$ mapping to the given two, showing connectedness of $(\Psi \downarrow (A,a))$, so $\Psi$ is coinitial. Thus, we may compute limits over $\mathcal I$ as limits over $\prod_{i=1}^n \el(F_{x_i})$, so \eqref{Eq representable limit} gives
\[
h_{\coprod x} = \coprod_{i=1}^n\ \lim_{\substack{\longleftarrow \\ (A_i,a_i) \in \el(F_{x_i})}} \ev(A_i) = \lim_{\substack{\longleftarrow \\ (A_i,a_i)_i \in \prod_i \el(F_{x_i})}}\ \coprod_{i=1}^n \ev(A_i) = \lim_{\substack{\longleftarrow \\ (A,a) \in \mathcal I}} \ev(A),
\]
since cofiltered limits in $\ProFin$ commute with finite coproducts and $\ev$ preserves finite coproducts. The result follows since every $F \in \Fun\cts(\Pi_1(\mathscr C),\Fin)$ is opcompact in $\Fun\cts(\Pi_1(\mathscr C),\ProFin)$ by \autoref{Lem finite opcompact}.
\end{proof}

\begin{Cor}\label{Cor filtered colimit}
Let $\mathscr C$ be a Galois category and let $\{F_i\ |\ i \in I\}$ be a finite conservative family of fibre functors.
\begin{enumerate}
\item\label{item ev surjection} If $F \in \Fun\cts(\Pi_1(\mathscr C),\Fin)$, there exists $A \in \mathscr C$ and a surjection $\ev(A) \twoheadrightarrow F$.
\item\label{item locally full} If $A, B \in \mathscr C$ and $\alpha \colon \ev(A) \to \ev(B)$ is a natural transformation, there exists an effective epimorphism $f \colon C \twoheadrightarrow A$ and a map $g \colon C \to B$ such that $\ev(g) = \alpha \circ \ev(f)$.
\item\label{item subobject} If $A \in \mathscr C$ and $F \subseteq \ev(A)$ is a subpresheaf, then there exists a monomorphism $I \hookrightarrow A$ identifying $F$ with $\ev(I)$.
\end{enumerate}
\end{Cor}

\begin{proof}
Given $F \in \Fun\cts(\Pi_1(\mathscr C),\Fin)$, there is a canonical map $\coprod_{x \in \Pi_1(\mathscr C)} F(x) \times h_x \twoheadrightarrow F$ in $\Fun\cts(\Pi_1(\mathscr C),\ProFin)$. Then \autoref{Lem pointed filtered} implies that it factors as $h_{\coprod x} \to \ev(A) \to F$ for some $A \in \mathscr C$, which proves \ref{item ev surjection}.

If $\alpha \colon \ev(A) \to \ev(B)$ is any natural transformation, there is again a canonical map $\coprod_{x \in \Pi_1(\mathscr C)} \ev(A)(x) \times h_x \twoheadrightarrow \ev(A) \to \ev(B)$, turning $A$ and $B$ into objects $(A,a)$ and $(B,b)$ of $(\{h_{\coprod x}\} \downarrow \ev)$. Setting $F = \ev(B)$, both $\ev(A) \to \ev(B)$ and $\ev(B) = \ev(B)$ are elements on the left hand side of \eqref{Eq pointed filtered} inducing the same map on the right hand side. Thus, there exists $(C,c) \in (\{h_{\coprod x}\} \downarrow \ev)$ with maps $f \colon C \to A$ and $g \colon C \to B$ such that $\alpha \circ \ev(f) = \ev(g)$ and $f(c) = a$ and $g(c) = b$. Since we pointed at all points of $\ev(A)$, this implies that $\ev(f)$ is surjective, hence $f$ is an effective epimorphism by \autoref{Lem Galois elementary}\ref{item effective epi}. This finishes the proof of \ref{item locally full}.

If $F \subseteq \ev(A)$ is a subpresheaf, then by \ref{item ev surjection}, there exists a surjection $\ev(B) \twoheadrightarrow F$, and by \ref{item locally full}, there exists an effective epimorphism $C \twoheadrightarrow B$ such that $\ev(C) \twoheadrightarrow \ev(B) \twoheadrightarrow F \hookrightarrow \ev(A)$ comes from a map $f \colon C \to A$. We may write $f$ as a composition $C \twoheadrightarrow I \hookrightarrow A$ of an effective epimorphism followed by a monomorphism, and these properties are preserved by $\ev$. We conclude that $\ev(I) = F$ as subpresheaves of $\ev(A)$.
\end{proof}

\begin{Thm}\label{Thm Galois exodromy}
Let $\mathscr C$ be a Galois category and let $\{F_i\ |\ i \in I\}$ be a finite jointly conservative family of fibre functors. Then the functor $\ev \colon \mathscr C \to \Fun\cts(\Pi_1(\mathscr C),\Fin)$ is an equivalence of categories.
\end{Thm}

Although the strategy above is very much inspired by \cite[Tag \href{https://stacks.math.columbia.edu/tag/03A0}{03A0}]{Stacks}, it does not rely on it.

\begin{proof}
We already saw in \autoref{Lem Galois elementary}\ref{item faithful} that $(F_i)_{i \in I} \colon \mathscr C \to \Fin^I$ is faithful. Since it factors through $\ev$, we see that $\ev$ is faithful. If $\alpha \colon \ev(A) \to \ev(B)$ is a map, then by \autoref{Cor filtered colimit}\ref{item locally full}, there exists an effective epimorphism $f \colon C \twoheadrightarrow A$ and a map $g \colon C \to B$ such that $\ev(g) = \alpha \circ \ev(f)$. If $R \rightrightarrows C$ is the kernel pair of $f$, then $A$ is the coequaliser of $R \rightrightarrows C$. The compositions $R \rightrightarrows C \overset g\to B$ agree after applying $\ev$, hence they agree by faithfulness, so $g \colon C \to B$ factors uniquely as $hf$ for some $h \colon A \to C$. We see that $\ev(h)\ev(f) = \ev(g) = \alpha \ev(f)$, so $\ev(h) = \alpha$ since $\ev(f)$ is surjective, proving that $\ev$ is full.

If $F \in \Fun\cts(\Pi_1(\mathscr C),\Fin)$, there exists a surjection $\ev(A) \twoheadrightarrow F$ by \autoref{Cor filtered colimit}\ref{item ev surjection}. If $R \subseteq \ev(A) \times \ev(A)$ is the kernel pair, there exists a subobject $B \subseteq A \times A$ such that $\ev(B) = R$ by \autoref{Cor filtered colimit}\ref{item subobject}. By fully faithfulness, the subobject $B \subseteq A \times A$ is an equivalence relation: it suffices to check this for $\Hom(-,B) \subseteq \Hom(-,A \times A)$, which can be computed after applying $\ev$. Thus, we may form the quotient $A \twoheadrightarrow C$, which is preserved by $\ev$, so $\ev(C) \cong F$. This proves essential surjectivity.
\end{proof}

\begin{Def}
Let $X$ be a qcqs scheme and let $X = \coprod_{s \in S} X_s$ be a finite decomposition into constructible locally closed subschemes. Assume that each stratum is connected, and choose a geometric point $\bar s$ in each stratum $X_s$, inducing a fibre functor $F_{\bar s} \colon \Sh\cons[S](X_\et) \to \Fin$. The \emph{stratified fundamental category} $\Pi_1^\et(X,S)$ is the fundamental category $\Pi_1(\mathscr C)$ of the pretopos $\mathscr C = \Sh\cons[S](X_\et)$ with the fibre functors $F_{\bar s}$ for $s \in S$.
\end{Def}

\begin{Cor}[Exodromy correspondence]\label{Thm exodromy}
Let $X$ be a qcqs scheme and let $X = \coprod_{s \in S} X_s$ be a finite decomposition into connected constructible locally closed subschemes. Then the natural functor
\[
\ev \colon \Sh\cons[S](X_\et) \to \Fun\cts\big(\Pi_1^\et(X,S),\Fin\big)
\]
is an equivalence.\hfill\qedsymbol
\end{Cor}

\begin{Par}
Similarly, \autoref{Ex scheme with cover} gives a variant for the $S$-constructible sheaves that are trivialised on a given set of connected pro-\'etale Galois covers $Y_s \to X_s$. Applying this to the tame universal cover (i.e.\ the maximal tamely ramified extension) gives an equivalence
\[
\Sh\cons[S]^t(X_\et) \stackrel\sim\to \Fun\cts\big(\Pi_1^t(X,S),\Fin\big).
\]
\end{Par}

\phantomsection
\bibliographystyle{alphaurledit}
{\footnotesize\bibliography{Exodromy.bib}}

\end{document}